\documentclass[11pt]{amsart}

\usepackage[T1]{fontenc}
\usepackage{lmodern}
\usepackage{microtype}
\usepackage{amsmath,amssymb,mathtools}
\usepackage{enumitem}
\usepackage{aliascnt}
\usepackage[hidelinks]{hyperref}

\setlist[enumerate]{leftmargin=*,label=(\roman*)}

\usepackage[a4paper, margin=3cm]{geometry}
\newcommand{\R}{\mathbb{R}}
\newcommand{\Z}{\mathbb{Z}}
\newcommand{\MRR}{M_\R}
\newcommand{\conv}{\operatorname{conv}}

\newcommand{\lin}{\operatorname{lin}}
\newcommand{\relint}{\operatorname{relint}}
\newcommand{\codeg}{\operatorname{codeg}}
\newcommand{\mcodeg}{\operatorname{mcodeg}}
\newcommand{\mdeg}{\operatorname{mdeg}}
\newcommand{\MV}{\operatorname{MV}}
\newcommand{\vertices}{\operatorname{Vert}}

\newcommand{\cayley}{\mathbin{*}}
\newcommand{\dual}[1]{{#1}^{\times}}
\newcommand{\MRbar}{{\overline M}_{\R}}
\newtheorem{theorem}{Theorem}[section]
\newcommand{\Est}{\operatorname{E_{st}}}

\newcommand{\pro}[2]{\langle #1, #2 \rangle}

\newaliascnt{proposition}{theorem}
\newtheorem{proposition}[proposition]{Proposition}
\aliascntresetthe{proposition}

\newaliascnt{lemma}{theorem}
\newtheorem{lemma}[lemma]{Lemma}
\aliascntresetthe{lemma}

\newaliascnt{corollary}{theorem}
\newtheorem{corollary}[corollary]{Corollary}
\aliascntresetthe{corollary}

\theoremstyle{definition}
\newaliascnt{definition}{theorem}
\newtheorem{definition}[definition]{Definition}
\aliascntresetthe{definition}

\theoremstyle{remark}
\newaliascnt{remark}{theorem}
\newtheorem{remark}[remark]{Remark}
\aliascntresetthe{remark}

\theoremstyle{remark}
\newaliascnt{example}{theorem}
\newtheorem{example}[example]{Example}
\aliascntresetthe{example}

\usepackage[nameinlink,capitalise,noabbrev]{cleveref}

\crefname{theorem}{Theorem}{Theorems}
\Crefname{theorem}{Theorem}{Theorems}
\crefname{proposition}{Proposition}{Propositions}
\Crefname{proposition}{Proposition}{Propositions}
\crefname{lemma}{Lemma}{Lemmas}
\Crefname{lemma}{Lemma}{Lemmas}
\crefname{corollary}{Corollary}{Corollaries}
\Crefname{corollary}{Corollary}{Corollaries}
\crefname{definition}{Definition}{Definitions}
\Crefname{definition}{Definition}{Definitions}
\crefname{remark}{Remark}{Remarks}
\Crefname{remark}{Remark}{Remarks}

\title[Decomposing Gorenstein polytopes of large index]
{Decomposing Gorenstein polytopes of large index}
\author{Johannes Knupfer}
\author{Benjamin Nill}
\address{Faculty of Mathematics, Otto-von-Guericke University Magdeburg, Universit\"atsplatz 2, 39106 Magdeburg, Germany}
\email{\{johannes.knupfer, benjamin.nill\}@ovgu.de}
\date{}

\begin{document}

\begin{abstract}
In this paper we prove a free-join decomposition theorem on Gorenstein polytopes significantly improving upon results in prior joint work of the second author with Batyrev, Borger, Haase, Kretschmer, and Payne. In particular, as a strong extension of the Batyrev-Juny lattice pyramid theorem it implies that any $d$-dimensional Gorenstein polytope $P$ of index larger than $\frac{d+2}{2}$ is a free join of Gorenstein polytopes. Here, the index (also called codegree) is the dilation factor $r$ such that $rP$ is reflexive. As our main application, we prove that the stringy $E$-polynomial of a Gorenstein polytope $P$ vanishes if and only if $P$ is thin (i.e., its local $h^*$-polynomial vanishes), and it has the expected degree otherwise. The latter was conjectured by Batyrev and the second author.

The proofs of the general results were found via ChatGPT 5.6 Sol, while a proof of the extremal case of the Batyrev-Juny conjecture is contained in the master thesis of the first author. 
\end{abstract}

\maketitle

\section{Introduction and main applications}

In this non-technical section we describe the main applications of our decomposition theorem (Theorem~\ref{main}) on Gorenstein polytopes. Proofs will be given in the next section. 

\subsection{Our setting}

Let $M\cong\Z^d$ be a lattice and denote by $\MRR \cong \R^d$ the associated real vector space. A polytope $P\subset M_{\R}$ is called {\em lattice polytope} if its vertices are in $M$. Two lattice polytopes are {\em isomorphic} or {\em unimodularly equivalent} if there is an affine-linear isomorphism from one polytope to the other that maps the lattice $M$ bijectively onto itself. In 1994 Batyrev \cite{batyrev94} introduced in the context of mirror symmetry the following class of lattice polytopes. A $d$-dimensional lattice polytope $P$ is {\em reflexive}\footnote{Classically, the unique interior lattice point is assumed to be the origin.} if it contains an interior lattice point $m \in M$ that has lattice distance one from every facet of $P$ (i.e, for each facet $F$ there is a linear functional $u$ in the dual lattice of $M$ such that $\pro{u}{f}-\pro{u}{m}=1$ for all $f \in F$). Subsequently in \cite{batyrevborisov,batyrev_mirror_1996}, together with Borisov, this definition was extended to a larger class of lattice polytopes. We say a lattice polytope $P$ is {\em Gorenstein} if there exists a positive integer $r$ such that $r P$ is reflexive. Here, the number $r$ is unique and is called the {\em index} or {\em codegree} of the Gorenstein polytope $P$. The index of a Gorenstein polytope is a fundamental invariant of a Gorenstein polytope and ranges between $1$ and $\dim(P)+1$. For each dimension $d$ there are only finitely many $d$-dimensional Gorenstein polytopes up to isomorphism. For instance for $d=1$, there is the reflexive polytope $[-1,1]$ (Gorenstein of index $1$) and the Gorenstein polytope $[0,1]$ of index $2$. Gorenstein polytopes satisfy beautiful duality properties and appear naturally in combinatorial commutative algebra, Ehrhart theory and related areas. We refer to \cite{BN08,NS13,Nil24} for more on the background and the basic results on Gorenstein polytopes.

\subsection{Free joins}
There is a straightforward way to get high-dimensional Gorenstein polytopes. For this let $P_1, P_2$ be Gorenstein polytopes with respect to the lattices $M_1$ and $M_2$. Then
\[P_1 \circ_\Z P_2 := \conv(P_1 \times \{0\} \times \{0\},\{0\} \times P_2 \times \{1\}) \subset (M_1)_\R \oplus_\R (M_2)_\R 
\oplus_\R \R\]
is called the {\em free join} of $P_1$ and $P_2$. It is again a Gorenstein polytope. Note that lattice polytopes $P$ that are isomorphic to free joins with a lattice point are often called {\em lattice pyramids}, i.e., $P \cong \conv(0,F)$, where $F$ is a facet of $P$ and $0$ a vertex of $P$ of lattice distance one\footnote{This means that there is an integral functional that evaluates to $1$ on $F$.} from $F$.

{\em Notational convention:} As in \cite{Nil24}, a lattice point is also considered a lattice pyramid.

Of course, most Gorenstein polytopes are not free joins. However, it follows from our main result that this is always the case if the index is roughly in the upper half of its possible range.

\begin{theorem}
    If $P$ is a $d$-dimensional Gorenstein polytope of index $r > \frac{d+2}{2}$, then $P$ is isomorphic to a free join of Gorenstein polytopes.
    \label{index}
\end{theorem}

The strict inequality in Theorem~\ref{index} is optimal. For example, let us consider the standard simplex $\Delta_n := \conv(0, e_1, \ldots, e_n) \subset \R^n$. Then $\Delta_n\times\Delta_n$ is a Gorenstein polytope of dimension $d=2n$ with index $r=n+1=\frac{d+2}{2}$ that is not a free join, see \cite[Remark~2.9]{Nil24}.

\subsection{Reformulation via degree and the Batyrev-Juny conjecture}

In order to understand the context and the proof of Theorem~\ref{index}, it is instructive to reformulate the index of a Gorenstein polytope using notation from Ehrhart theory (cf. \cite{batyrev_multiples_2007}).

Given a lattice polytope $P$ of dimension $d$ its {\em codegree} and {\em degree} are defined as follows:
\[
  \codeg P:=\min\{r\geq1:\relint(rP)\cap M\neq\varnothing\},
  \qquad
  \deg P:=d+1-\codeg P.
\]

It follows from Ehrhart-theoretic reciprocity that $\codeg(P) \in \{1, \ldots, d+1\}$ and $\deg(P) \in \{0, \ldots, d\}$. The degree of a lattice polytope equals the degree of its $h^*$-polynomial, the numerator polynomial of the rational generating series of its Ehrhart polynomial. As such, it is a fundamental invariant of a lattice polytopes that can be regarded as its `lattice dimension', and thus there has been lots of study on lattice polytopes of given degree. 

Note that for a Gorenstein polytope the codegree is precisely its index. Hence, Theorem~\ref{index} can be reformulated as follows: {\em Gorenstein polytopes of dimension $d$ and degree $s$ with $d > 2s$ are free joins of Gorenstein polytopes.} This proves Conjecture~2.8 in \cite{Nil24}, called refined Batyrev-Juny conjecture. The name stems from the following result conjectured by Batyrev and Juny in \cite{BJ10} and proven by the second author in \cite{Nil24}.

\begin{theorem}
    If $P$ is a $d$-dimensional Gorenstein polytope of degree $s$ with $d \ge 3s$, then $P$ is a lattice pyramid.\label{bj}
\end{theorem}

As already shown in \cite[Prop.~2.12]{Nil24} Theorem~\ref{bj} is an easy consequence of Theorem~\ref{index}.

Batyrev and Juny also conjectured that there is up to isomorphisms only one Gorenstein polytope of dimension $d=3s-1$ and degree $s$ which is not a lattice pyramid. This part of the conjecture was left open in \cite{BJ10}. It was first proven in the Masterthesis \cite{Knu25} of the first author. 

\begin{corollary}
Let $P$ be a Gorenstein polytope of dimension $3s-1$ and degree $s>0$. Then $P$ is not a lattice pyramid if and only if $P$ is isomorphic to the free join of $s$ copies of $[0,1]^2$.
\label{equality}
\end{corollary}

In fact, this is a special case of the following generalization (cf. \cite{Knu25}). For this, let us define for $k \in \Z_{\ge1}$, the set $\mathcal{P}_k$ as the finite set of representatives for the isomorphism classes of Gorenstein polytopes $P$ of degree $1 \le s\le k$ and dimension $3s-k$ that are not lattice pyramids.

\begin{corollary}
    Any Gorenstein polytope $P$ of degree $s$ and dimension $3s-k$ with $k \ge 1$ that is not a lattice pyramid is isomorphic to a polytope in $\mathcal{P}_k$ or to a free join of polytopes in $\mathcal{P}_{k'}$ with $1\le k' \le k$.
    \label{close}
\end{corollary}

Note that this proves Corollary~\ref{equality}. One only has to observe that $[0,1]^2$ is the only Gorenstein polygon of degree one that is not a lattice pyramid, so we have $\mathcal{P}_1=\{[0,1]^2\}$. 

\subsection{Applications to thin Gorenstein polytopes and the stringy $E$-polynomial}

Motivated by mirror symmetry of Calabi-Yau varieties, the notion of {\em stringy $E$-polynomials} of Gorenstein polytopes was introduced by Batyrev and the second author in \cite{BN08}. In that article, several conjectures on this polynomial were phrased, the most basic one regarding its expected degree. In \cite[Section~6]{NS13} an explicit approach to proving this question was described by Schepers and the second author. A partial result in this direction was later achieved in \cite[Corollary~6.12]{borger_thin_2023}, which was however not yet sufficient. Here, we can give now a completely satisfactory answer. For this, it is natural to consider the notion of thin polytopes.

In local Ehrhart theory, not only the classical $h^*$-polynomial of a lattice polytope but also a more intricate version, called its local $h^*$-polynomial, is being studied. It has a direct interpretation as the Hodge vector of a lattice polytope, see \cite{borisov_string_2003,preserving}, and plays a crucial part in the definition of stringy $E$-polynomials \cite{BN08}. A lattice polytope is called {\em thin} if the local $h^*$-polynomial is zero. The study of thin polytopes has been originated for simplices by Gelfand, Kapranov and Zelevinsky 
\cite{GKZ}, and has been continued for polytopes in \cite{borger_thin_2023}. In the last paper also a characterization of thin Gorenstein polytopes was given \cite[Theorem~6.3]{borger_thin_2023}. By strengthening this characterization in Theorem~\ref{thin-char} in the next section, the following result can be obtained.

\begin{theorem}
Let $P$ be a $d$-dimensional Gorenstein polytope of index $r$. Then 
\[\Est(P;u,v) = 0 \; \text{ if and only if } \; P \text{ is thin.}\]
Otherwise, the stringy $E$-polynomial $\Est(P;u,v)$ of $P$ has degree $2n$, where $n:=d+1-2r$ is the {\em Calabi-Yau dimension} of $P$.
\label{thin-thm}
\end{theorem}

This proves the first part of Conjecture~4.10 in \cite{BN08}, see also Conjecture~3.3(2) in \cite{NS13}, and also answers positively Question~6.3(1) in \cite{NS13}.

\medskip

The paper is organized as follows. In Section 2 we present our main decomposition theorem, describe all necessary notions such as degree joins and centrally thin polytopes, and give the proofs of the results stated above. Sections 3 and 4 contain the proof of the main result. In Section 3 we analyze the Cayley result from \cite{HNP09} in detail which will eventually lead to the existence of long degree joins. In Section 4 it will be shown that in our situation they even split as free joins.

\subsection*{Acknowledgment} An independent purely-human proof of Corollary~\ref{equality} is contained in the Masterthesis \cite{Knu25} of the first author. The general results and a first version of the paper were obtained using ChatGPT 5.6 Sol using essentially the same ingredients (Cayley polytopes, Gorenstein joins, nonnegativity of the mixed degree) as in \cite{Nil24} by the second author and in the master thesis \cite{Knu25} of the first author supervised by the second author, but combining them in a fresh way. This final version has been largely rewritten and restructured. For some editing purposes also ChatGPT 6 Astra was used. Some new results, new notions and comments were manually added and incorporated to distill the essential features: a weaker notion of join that we call degree join (see Section~2), a successive use of the Cayley decomposition theorem from \cite{HNP09} (see Section~3), and the fact that splitting is induced from vertex-spanning dual Gorenstein polytopes (see Section~4). The last point has also already appeared in the Masterthesis \cite{Knu25}. This work is funded by the Deutsche Forschungsgemeinschaft (DFG, German Research Foundation) – 539867500 as part of the research priority program Combinatorial Synergies. 

\section{The main result}

In this section we describe the notions needed for our main decomposition result and give the proofs of its applications.

\subsection{Cayley polytopes}

For lattice polytopes $P_0,\ldots,P_l\subset M_{\R}$, their {\em Cayley polytope} is
\[
  P_0\cayley\cdots\cayley P_l
  :=\conv(P_0\times e_0,\ldots,P_l\times e_l).
\]
It is a lattice polytope of dimension
\begin{equation}\label{eq:cayley-dimension}
  \dim(P_0\cayley\cdots\cayley P_l)=\dim(P_0+\cdots+P_l)+l.
\end{equation}

There is a well-known reformulation of Cayley polytopes. For this, let us recall that a lattice simplex $P$ of dimension $d$ is called {\em unimodular} if it is isomorphic to $\conv(0,e_1,\ldots,e_d)$. A lattice polytope is a unimodular simplex if and only if its degree is zero. In particular, degree $0$ implies lattice pyramid.

Now, a lattice polytope $P$ is isomorphic to a Cayley polytope of $l+1$ lattice polytopes if and only if it admits a lattice projection\footnote{A {\em lattice projection} $P \subset \MRR$ onto $P' \subset M'_{\R}$ is an affine-linear map that maps the lattice $M$ surjectively onto $M'$.} onto a unimodular $l$-simplex. In this case, we say $P$ has a {\em Cayley decomposition} with factors $F_0, \ldots, F_l$ where these are the fibers over the vertices of the unimodular simplex \cite[Proposition~2.3]{BN08}. Note that necessarily $F_0, \ldots, F_l$ are disjoint faces of $P$ with $P=\conv(F_0, \ldots, F_l)$. We will also recall a useful dual characterization from \cite{BN08} in Section~3.2.

The following result (\cite[Theorem~3.1]{HNP09}) shows that Gorenstein polytopes of large index admit Cayley decompositions. Note that for $d >2s$ this will be drastically improved by the main decomposition theorem.

\begin{theorem}
Let $P$ be a Gorenstein polytope of dimension $d$ and degree $s>0$. If $d \ge 2s$, then $P$ is a Cayley polytope of lattice polytopes in dimension $\le 2s-1$.\label{cayley}
\end{theorem}

\subsection{Degree joins of lattice polytopes}

Let us consider several notions of joins. A lattice polytope $P$ is a {\em join} of faces $F_0, \ldots, F_l$ if $P = \conv(F_0, \ldots, F_l)$ and
\[\dim(P) = \dim(F_0) + \cdots + \dim(F_l) + l.\]
Equivalently, there is an affine-linear isomorphism (not necessarily, a unimodular equivalence) between $P$ and $F_0 \circ_\Z \cdots \circ_\Z F_l$.

If $P$ is a join of faces $F$ and $G$ such that there is also a Cayley decomposition of $P$ with factors $F$ and $G$, then $P$ is called {\em Cayley join} of the faces $F$ and $G$. 

Here, is a notion that has not yet been defined in the literature but will turn out to give just the right level of flexibility needed later.

\begin{definition}
    A  lattice polytope $P$ is an {\em degree join} of faces $F_0, \ldots, F_l$ if $P$ is a join of $F_0, \ldots, F_l$ and 
    \[\deg(P)=\deg(F_0) + \cdots + \deg(F_l),\]
    or equivalently
    \[\codeg(P)=\codeg(F_0) + \cdots + \codeg(F_l).\]
\end{definition}

Clearly, free joins are degree joins, see \cite[Remark~4.6]{NS13}.

For Gorenstein polytopes there is an even stronger notion (this is \cite[Lemma~5.13]{borger_thin_2023}). A Gorenstein polytope $P$ is a {\em Gorenstein join} of faces $F$ and $G$ if it is a Cayley join as well as a degree join of $F$ and $G$. We say a Gorenstein polytope is {\em reducible} if it is a Gorenstein join, and {\em irreducible} otherwise. We remark that if a Gorenstein polytope is a free join, then it is also a Gorenstein join (see \cite[Section~5]{borger_thin_2023}). Note that by convention a lattice point is reducible.

\begin{remark}
    Note that while given two arbitrary lattice polytopes $P_1, P_2$ the notion of the free join of $P_1, P_2$ is well-defined, the notion of a join, Cayley join, degree join, or Gorenstein join of $P_1,P_2$ is {\em not}. For a lattice polytope $P$ to be a join, Cayley join, degree join, or Gorenstein join of two faces $P_1,P_2$ is a {\em property} of $P$ -- not a construction. In particular, we do not have an `intrinsic lattice' for each factor but we always consider the ambient lattice of $P$ restricted to the affine hull of the factor.
\end{remark}

\begin{remark}
    It is important to note that while joins and free joins are `associative', Cayley joins and Gorenstein joins are not. For Cayley joins an example was given in \cite[Example~5.17]{borger_thin_2023}. For Gorenstein joins this is shown by the following example: Let $P \subset \R^6$ be given as the convex hull of the subsets $A,B,C$ defined as follows:
    \[\begin{aligned}
A&=\operatorname{conv}\{
(0,1,-1,0,-1,0),\,
(0,-1,-1,0,-1,0)\},\\[1mm]
B&=\operatorname{conv}\{
(-1,0,1,0,0,1),\,
(1,0,-1,0,0,1)\},\\[1mm]
C&=\operatorname{conv}\{
(-1,1,0,1,0,0),\,
(1,-1,0,-1,0,0), (0,0,0,-1,1,1),\,
(0,0,0,1,-1,-1)\}.
\end{aligned}\]
Then $A,B,C$ are reflexive polytopes (up to lattice translations), the first two are simply reflexive intervals, while the last one is a two-dimensional reflexive crosspolytope. Let $F := \conv(A,B)$. Then $F$ is free join of $A$ and $B$, and in particular Gorenstein. Moreover, $P$ is a Gorenstein join of $F$ and $C$. On the other hand, if one defines $G := \conv(B,C)$, then again $G$ is a free join of $B$ and $C$, and thus Gorenstein. However, $P$ is {\em not} a Cayley join, and thus not a Gorenstein join, of $A$ and $G$. We leave the computational verification of these statements to the reader.
\end{remark}

One nice feature of degree joins is indeed their associativity.

\begin{proposition}
Let $P$ be a lattice polytope with faces $F_0, \ldots, F_l$. Set $G := \conv(F_1, \ldots, F_l)$. Then the following two statements are equivalent:
\begin{enumerate}
    \item $P$ is a degree join of faces $F_0, \ldots, F_l$,
    \item $P$ is a degree join of the faces $F_0$ and $G$, and $G$ is a degree join of the faces $F_1, \ldots, F_l$.
\end{enumerate}
If this holds and $P$ is a Gorenstein polytope, then also $F_0, \ldots, F_l$ are Gorenstein polytopes.
\label{associative}
\end{proposition}

The proof can be found in Subsection~\ref{subsec:splitting}.

\begin{remark}
    The converse of the last statement of the previous proposition is wrong. This can be seen from the following example. Let 
    \[
P=
\operatorname{conv}\left\{
(-1,0,1),\,
(1,0,1),\,
(-1,-2,0),\,
(1,2,0)
\right\}.
\]It has the two edges
\[
F_0=
\operatorname{conv}\{(-1,0,1),(1,0,1)\}
\]and
\[
F_1=
\operatorname{conv}\{(-1,-2,0),(1,2,0)\}.
\]
One can directly verify that $P$ is a tetrahedron of codegree $2$ that is a join of $F_0$ and $F_1$, both being reflexive polytopes (of codegree $1$). So, $P$ is a degree join with two Gorenstein factors, however, $P$ is not Gorenstein as $2P$ has five interior lattice points, so it is not reflexive.
\end{remark}

The crucial property of Gorenstein polytopes is their duality \cite{BN08}. Each Gorenstein polytope $P$ has a dual Gorenstein polytope $P^\times$ of the same dimension, degree and codegree (index) that is combinatorially dual to $P$. We will recall this construction again in the next section. Moreover, $P$ is irreducible if and only if $P^\times$ is irreducible \cite{NS13}. 

\begin{definition}\label{def:vertex-spanning}
We call a lattice polytope $P \subset \MRR$ \emph{vertex-spanning} if every lattice point in its affine hull is an affine integer combination of its vertices. Equivalently, assuming $P$ is full-dimensional,
\[
  \sum_{v\in\vertices(P)}\Z(v,1)= M \oplus_\Z \Z.
\]
\end{definition}

The following splitting lemma will be essential for getting free joins in the main decomposition theorem.

\begin{lemma}
Let $P$ be a Gorenstein polytope that is a degree join of faces $F_0, \ldots, F_l$ (for $l > 0$). If $F_0$ has a vertex-spanning dual Gorenstein polytope $F_0^\times$, then $P$ is a free join of $F_0$ with $G := \conv(F_1, \ldots, F_l)$. Here, $G$ is itself a Gorenstein polytope that is a degree join of $F_1, \ldots, F_l$.
    \label{splitting}
\end{lemma}

Also this proof will be given in Subsection~\ref{subsec:splitting}.

As a nice consequence we get the following statement generalizing \cite[Corollary~5.16]{borger_thin_2023}.

\begin{lemma}
    If a Gorenstein polytope $P$ is a degree join, where one factor is a lattice pyramid (e.g., has degree $0$), then $P$ is a lattice pyramid.
    \label{pyramid}
\end{lemma}

\begin{proof}
    Let $P$ be a degree join of faces $F$ and $F_2$, where $F$ is a lattice pyramid. If $F$ is a lattice point, then Lemma~\ref{splitting} yields that $P$ is a free join of $F$ and $F_2$, so a lattice pyramid. Otherwise, $F$ is a free join of a lattice point $F_0$ and a face $F_1$. By Proposition~\ref{associative} $P$ is a degree join of $F_0$ and $G$, where $G := \conv(F_1,F_2)$. As $F^\times_0$ also is a lattice point, and in particular vertex-spanning, Lemma~\ref{splitting} implies that $P$ is a free join of $F_0$ and $G$, and thus also a lattice pyramid.
\end{proof}

We remark that this statement needs that $P$ is Gorenstein. For instance, the triangle $\conv((0,0),(2,0),(0,2))$ is a degree join of $\conv((0,0),(2,0))$ and $\{(0,2)\}$, however not a lattice pyramid.

\subsection{Centrally thin Gorenstein polytopes}

In our main decomposition theorem, the following special class of Gorenstein polytopes plays a crucial role.

\begin{definition}
    A Gorenstein polytope $P$ of degree $s$ and dimension $d=2s$ is called {\em centrally thin} if $P$ is irreducible when $d \ge 1$.
\end{definition}

By convention, lattice points are centrally thin. Let us remark that lattice polytopes of dimension $d$ and degree $s$ with $d \ge 2s$ are called {\em trivially thin} in \cite{borger_thin_2023}, as they are automatically thin. Moreover, by Theorem~\ref{cayley} any centrally thin Gorenstein polytope of positive dimension is a Cayley polytope.

\begin{remark}
    Let $P$ be centrally thin of dimension $d$ and degree $s$. If $s=0$, then $d=0$ and $P$ is a lattice point. If $s=1$, then $d=2$ and $P \cong [0,1]^2$ by the classification of thin polygons\footnote{Note that the exceptional triangle in the classification in  \cite{batyrev_multiples_2007} (the second dilate of the unimodular triangle) is not Gorenstein.}. If $s=2$, then $d=4$ and $P$ is isomorphic to one of the Gorenstein polytopes $Q_1,Q_2,Q_4,Q_5$ (where $Q_4$ is dual to $Q_1$ and $Q_5$ is dual\footnote{Note that the diagram of $Q_2$ in \cite{BJ10} is wrong. It should be $\conv(0,e_1,e_2)*\conv(0,e_1)*\conv(0,e_2)$.} to $Q_2$) in the classification of Gorenstein polytopes of degree $2$ by Batyrev and Juny \cite{BJ10}.
    \label{ct-classification}
\end{remark}

As being irreducible as well as dimension and degree of Gorenstein polytopes are invariant under duality, let us observe that the dual of a centrally thin Gorenstein polytope is again centrally thin.

The following result allows for an algorithmic computation of centrally thin Gorenstein polytopes via classifying Minkowski summands of reflexive polytopes.

\begin{proposition}
Let $P_0, \ldots, P_s$ be lattice polytopes in $\R^s$ (with $s > 0$) such that 
\begin{enumerate}
    \item all $P_0, \ldots, P_s$ are positive-dimensional,
    \item all $P_0, \ldots, P_s$ contain $0$,
    \item no proper subsum $\sum_{i \in I} P_i$ (for $\emptyset \not= I \subsetneq \{0, \ldots, s\}$) contains $0$ in its relative interior,
    \item $P_0 + \cdots + P_s$ is reflexive with $0$ in its interior.
\end{enumerate}
Then $P_0 * \cdots * P_s$ is a centrally thin Gorenstein polytope of degree $s$ (and dimension $2s$). Up to isomorphisms, any centrally thin Gorenstein polytope of degree $>0$ occurs in this way.\label{tuples}
\end{proposition}

The proof will be given in Subsection~\ref{subsec:centrally thin}.

\begin{remark}
    The condition on the tuple is precisely what it means to be a {\em proper centered irreducible nef-partition} in $\R^s$ of length $s+1$ \cite{batyrevborisov, BN08}. 
    \end{remark}

\begin{remark}
In fact, by Lemma~\ref{lem:hollow} condition (iii) can be replaced by condition
\begin{enumerate}
    \item[(iii)'] no proper subsum $\sum_{i \in I} P_i$ (for $\emptyset \not= I \subsetneq \{0, \ldots, s\}$) contains a lattice point in its relative interior
\end{enumerate}
In the notation of \cite{Nil20}, this means that $P_0, \ldots, P_s$ are families of mixed degree $0$ and length $s+1$. We refer to Section 2.2 in \cite{Nil20} for results and questions on such families of lattice polytopes. Let us give one example. If $P_0 = \conv(0,e_1)$, $P_1 = \conv(0,e_2)$ and $P_2=[-1,0]^2$, then $P_0,P_1,P_2$ satisfy the conditions of Proposition~\ref{tuples} and thus $P_0 * P_1 * P_2$ is a centrally thin Gorenstein polytope in dimension $4$, namely $Q_5$ in the classification in \cite{BJ10}.
\label{strengthen}
\end{remark}

Let us note one special property of centrally thin Gorenstein polytopes.

\begin{proposition}
     Let $P$ be a centrally thin Gorenstein polytope. Then $P$ and $P^\times$ are vertex-spanning.
\label{duality}
\end{proposition}

Also this proof will be given in Subsection~\ref{subsec:centrally thin}.

\subsection{The decomposition theorem and its applications}

The number $2s-d-1$ of a Gorenstein polytope is called its {\em Calabi-Yau dimension}. In this paper, we are interested in Gorenstein polytopes of negative Calabi-Yau dimension (or in the notation of \cite{borger_thin_2023} trivially thin Gorenstein polytopes). Therefore, its negative is more convenient to consider.

\begin{definition}
Let us write $q := d+1-2s$, and call it the {\em negative CY-dimension}.
\end{definition}

Note that if $r$ is the codegree of $P$, we have $q=r-s$. In particular, just as degree and codegree, the negative CY-dimension is additive for degree joins, and it also stays the same under the duality of Gorenstein polytopes.

Notice that $q \ge 1$ if and only if $d \ge 2s$. In particular, an irreducible Gorenstein polytope is centrally thin if and only if $q=1$.

With this preparation, we can formulate our main result.

\begin{theorem}
    Let $P$ be a Gorenstein polytope of dimension $d$ and negative CY-dimension~$q$. If $q \ge 1$, then
    \begin{enumerate}
        \item $P$ is centrally thin (and $q=1$), or $P$ is a free join of $q \ge 2$ centrally thin Gorenstein polytopes, or
        \item $P$ is a free join of $q$ centrally thin Gorenstein polytopes and one Gorenstein polytope of zero CY-dimension.
    \end{enumerate}
    In any case, $q=1$ or $P$ is a free join of $q \ge 2$ Gorenstein polytopes of negative CY-dimension one.
    \label{main}
\end{theorem}

In order to appreciate this result, let us prove all theorems presented in the introduction. Throughout, let $P$ be a Gorenstein polytope of dimension $d$ and degree $s$.

\begin{proof}[Proof of Theorem~\ref{index}]

The index condition translates to $d > 2s$ or equivalently $q \ge 2$. Theorem~\ref{main} implies that $P$ is a free join of at least two Gorenstein polytopes. 

\end{proof}

Let us also give a quick new proof of the Batyrev-Juny conjecture.

\begin{proof}[Proof of Theorem~\ref{bj}]

We may assume $s>0$. The condition $d \ge 3s$ translates to $q \ge s+1$. Theorem~\ref{main} implies that $P$ is a free join of at least $s+1$ Gorenstein polytopes. As the degrees of the factors must add up to $s$, we see that at least one of the factors must have degree zero. Lemma~\ref{pyramid} yields that $P$ is a lattice pyramid.

\end{proof}

Let us prove the generalization of the characterization in the Batyrev-Juny conjecture.

\begin{proof}[Proof of Corollary~\ref{close}]

We prove the statement by induction on the dimension. We may assume $s > k$. So, we have $d=3s-k>2s$ and the degree version of Theorem~\ref{index} implies that $P$ is a free join of Gorenstein polytopes $F$ and $G$ with dimensions $d_F$ and $d_G$, as well as degrees $s_F$ and $s_G$. We have $s=s_F+s_G$. As $P$ is not a lattice pyramid, this is also true for $F$ and $G$. Thus the Batyrev-Juny conjecture implies $d_F=3s_F-k_F$ and $d_G=3s_G-k_G$ for $k_F,k_G \ge 1$. Hence, $d_F+d_G+1=d=3s-k=3s_F+3s_G-k=d_F+k_F+d_G+k_G-k$, therefore
\[k_F+k_G=k+1.\]
This implies $k_F,k_G \le k$. Applying induction to $F$ and $G$ yields the result.
\end{proof}

For the applications to thin polytopes and the stringy $E$-polynomials, we need a new characterization of thin Gorenstein polytopes extending the one in \cite{borger_thin_2023}. The fact that we get free joins and not just Gorenstein joins in this result solves a problem that had to be left open in \cite{borger_thin_2023}. In the proof we use the novel notion of a degree join and the special properties of centrally thin Gorenstein polytopes.

\begin{theorem}
    Let $P$ be a Gorenstein polytope of dimension $>0$. Then $P$ is thin if and only if $P$ is centrally thin or $P$ is a free join of two Gorenstein polytopes with at least one factor centrally thin.
    \label{thin-char}
\end{theorem}

\begin{proof}

As centrally thin polytopes are (trivially) thin and the local $h^*$-polynomial is multiplicative with respect to free joins, the reverse implication is clear. So, let $P$ be thin. By \cite[Theorem~6.3]{borger_thin_2023} (a) $P$ is trivially thin or (b) $P$ is a Gorenstein join of two Gorenstein polytopes $F$ and $G$ with $F$ being trivially thin. Let us recall that trivially thin is equivalent to having negative CY-dimension, so Theorem~\ref{main} implies that (a) $P$ or (b) $F$ is centrally thin or a free join of two Gorenstein polytopes $F'$ and $G'$ with $F'$ being centrally thin. Assuming that $P$ is not centrally thin itself, we see from Proposition~\ref{associative} that in both cases $P$ is a degree join of two or three Gorenstein polytopes, where the first factor is a centrally thin Gorenstein polytope. By Proposition~\ref{duality}, Lemma~\ref{splitting} yields that $P$ is a free join of a centrally thin polytope with a Gorenstein polytope.
\end{proof}

By the characterization Proposition~\ref{tuples}, this allows for an effective computational approach towards classifying thin Gorenstein polytopes.

We get a new consequence from the fact that centrally thin Gorenstein polytopes are even-dimensional.

\begin{corollary}
    Odd-dimensional thin Gorenstein polytopes are free joins.
\end{corollary}

Now, the proof of the degree result on stringy $E$-polynomials follows from previous joint work with Schepers \cite{NS13}.

\begin{proof}[Proof of Theorem~\ref{thin-thm}]

 We may assume $d > 0$. Let $n := 2s-d-1$ be the CY-dimension of $P$. It follows from \cite[Corollary~3.7]{NS13} that the stringy $E$-polynomial vanishes if $n < 0$ or equivalently $d \ge 2s$. In particular, centrally thin Gorenstein polytopes have vanishing stringy $E$-polynomial. 

Now, let $P$ be thin. By Theorem~\ref{thin-char} and the multiplicativity of stringy $E$-polynomials for free joins, we see that $\Est(P;u,v)=0$. 

On the other hand, assume that $P$ is not thin. By \cite[Corollary~6.12]{borger_thin_2023} $P^\times$ is also not thin. By \cite[Proposition~6.9]{borger_thin_2023}, the degree of the local $h^*$-polynomial of $P^\times$ equals the degree of $P^\times$. Thus, by \cite[Lemma~6.1]{NS13}, the empty face $\emptyset$ contributes to the constant coefficient of $\Est(P;u,v)$. Hence, by Poincar\'e duality \cite[Proposition~3.2(2)]{NS13} and the polynomiality of the stringy $E$-polynomial (\cite[Corollary~3.7]{NS13} and \cite[Remark~3.4(1)]{NS13}) $\Est(P;u,v)$ has degree $2n$.

\end{proof}

\section{The Cayley descent}

The coarse idea of the proof is to analyze the proof of Theorem~\ref{cayley} about Cayley polytopes more precisely and then to apply this version successively. In order for this to work, the irreducibility assumption will be translated into a statement on tuples of lattice polytopes, and a different ingredient already used in \cite{Nil24}, namely, the nonnegativity of the mixed degree, will become crucial.

\subsection{Preliminaries on Gorenstein polytopes and reflexive Gorenstein cones}

Throughout, every lattice polytope is considered in the lattice of its affine
span, and every isomorphism is an affine lattice isomorphism.  Let us recall standard
facts about reflexive and Gorenstein polytopes as in \cite{BN08,NS13}.

Let us define $\overline M := M \oplus \Z$. The dual lattices of $M$ and $\overline M$ are denoted by $N$ and $\overline N$. For a lattice polytope $P\subset M_{\R}$, its {\em homogenized cone} is
\[
  C_P:=\R_{\geq0}(P\times\{1\})\subset \MRbar.
\]

A pointed rational cone $\sigma\subset \MRbar$ is a {\em Gorenstein cone} if its primitive ray generators lie in an affine hyperplane
$\langle n_{\sigma},-\rangle=1$ for a primitive
$n_{\sigma}\in\overline N$ (called the {\em degree functional} of $\sigma$). It is a {\em reflexive Gorenstein cone}
if its dual cone $\sigma^\vee$ is also Gorenstein.  In this case, if $m_{\sigma^{\vee}}\in\overline M$ is the degree functional of $\sigma^\vee$, the {\em index} of $\sigma$ (and $\sigma^\vee$) is
\[
  \langle n_{\sigma},m_{\sigma^{\vee}}\rangle.
\]
Now, the polytope $P$ is Gorenstein of index $r$ if and only if $C_P$ is reflexive Gorenstein of index $r$. The {\em support polytope} (i.e., the convex hull of the primitive ray generators) of $C_P^{\vee}$ is the {\em dual
Gorenstein polytope}, denoted by $\dual P$. Hence, Gorenstein duality preserves dimension, degree, and codegree.

\subsection{Analyzing Gorenstein Cayley polytopes more carefully}

Let us use the same notation as before. By \cite[Proposition~2.3]{BN08} it is known that each Cayley decomposition of $P$ in $l+1$ factors corresponds to writing $n_{C_P}$ as a sum of $l+1$ non-zero lattice points in $C_P^{\vee}$. More precisely, if $n_{C_P} = w_0 + \cdots + w_l$ for $w_0, \ldots, w_l \in (C_P^{\vee} \cap \overline N)\setminus\{0\}$, then there is a lattice projection $\overline M \to \Z^{l+1}$ given by 
\[x \mapsto (\pro{w_0}{x}, \ldots, \pro{w_l}{x})\]
mapping $P  \times \{1\}$ onto $\conv(e_0, \ldots, e_l)$. We call this the associated {\em Cayley presentation} with Cayley factors the preimages of the vertices.  Every Cayley decomposition can be described in this way. The image of a point $x \in \overline M_\R$ under this projection is called the associated {\em Cayley coordinates}. Note that their sum equals precisely the {\em height} $\pro{n_{C_P}}{x}$.

Suppose that $P=P_0\cayley\cdots\cayley P_l$ is Gorenstein of codegree $r$. The unique interior lattice point of $r(P\times\{1\})$ (namely, $m_{C_P^{\vee}}$) has positive integral Cayley
coordinates
\[
  (k_0,\ldots,k_l),
  \qquad \text{ with } 
  k_0+\cdots+k_l=r.
\]
We say, the Cayley presentation has {\em type} $(k_0,\ldots,k_l)$.

\begin{proposition}\label{prop:weighted-cayley}
Let $P_0, \ldots, P_l$ be lattice polytopes with $n:=\dim(P_0+\cdots+P_l)$. Let $(k_0,\ldots,k_l)\in\Z_{\geq1}^{l+1}$. Then $P_0 * \cdots * P_l$ is Gorenstein with Cayley presentation of type $(k_0,\ldots,k_l)$ if and only if
\begin{enumerate}
  \item $k_0P_0+\cdots+k_lP_l$ is reflexive up to translation;
  \item if $k_i\geq2$, then
  \[
    \dim\Bigl(\sum_{j\neq i}P_j\Bigr)<n.
  \]
\end{enumerate}
\end{proposition}

This is \cite[Proposition~3.1]{Nil24}. 

Using this notation, let us now make the Cayley decomposition result from \cite{HNP09} more explicit.
\begin{lemma}
\label{lem}
Let $P$ be a Gorenstein polytope of dimension $d$, degree $s>0$, and
codegree $r=d+1-s$.  If
\[
  q:=d+1-2s=r-s>0,
\]
then $P$ has a Cayley presentation of type $(s,1^{(q)})$.
\end{lemma}

Here, $(s,1^{(q)})$ stands for $(s,1,\ldots,1)$, with $q$ many $1$s.

\begin{proof}
Let $\sigma=C_P$, and let $n_{\sigma}$ and $m_{\sigma^{\vee}}$ be the
degree elements of $\sigma$ and $\sigma^{\vee}$, so $\langle m_{\sigma^{\vee}},n_{\sigma}\rangle=r$.

We follow the proof of \cite[Theorem~3.1]{HNP09}.  Choose a full-dimensional lattice simplex $S$ with vertices $v_0,\ldots,v_d$ in $P^\times$ whose cone contains $n_{\sigma}$.  Splitting
the coefficients of $n_{\sigma}$ into integral and fractional parts gives
\[
  n_{\sigma}=u_1+\cdots+u_N+\{n_{\sigma}\},
  \qquad
  \{n_{\sigma}\}=\sum_{i=0}^{d}b_iv_i,
  \quad 0\leq b_i<1,
\]
where each $u_j$ is a vertex of the simplex $S$ (not necessarily distinct). Note that as in \cite{HNP09} by the interpretation of the degree and its monotonicity properties, we get

\[\sum_{i=0}^d b_i\leq \deg(S)\le \deg(P^\times) = s.\] 
Every such vertex has height one with respect to
$m_{\sigma^{\vee}}$, hence
\[
  N=\pro{m_{\sigma^{\vee}}}{u_1+\cdots+u_N}=\pro{m_{\sigma^{\vee}}}{n_\sigma-\{n_{\sigma}\}}=r-\sum_{i=0}^d b_i\geq r-s=q.
\]
We set
\[
  w:=n_{\sigma}-u_1-\cdots-u_q.
\]
Clearly, $w\in\sigma^{\vee}\cap\overline N$.  Moreover,
\[
  \langle m_{\sigma^{\vee}},w\rangle=r-q=s>0,
\]
and therefore $w\neq0$.  We have obtained a decomposition
\[
  n_{\sigma}=w+u_1+\cdots+u_q
\]
into nonzero lattice points of $\sigma^{\vee}$, and thus a Cayley presentation of $P$ of type
\[
  \bigl(
  \langle m_{\sigma^{\vee}},w\rangle,
  \langle m_{\sigma^{\vee}},u_1\rangle,\ldots,
  \langle m_{\sigma^{\vee}},u_q\rangle
  \bigr)
  =(s,1,\ldots,1).
\]
\end{proof}

\subsection{Expanded family, irreducibility constraints and the mixed degree}

Quite naturally, families of lattice polytopes play an important role when considering lattice polytopes of small degree.

Following above notation, if a Gorenstein Cayley polytope $P\cong P_0 * \cdots * P_l$ has a Cayley presentation of type $(k_0,\ldots,k_l)$, its \emph{expanded family} is the family of lattice polytopes in which $P_0, \ldots, P_l$ occur with multiplicities $k_0, \ldots, k_l$. Its Minkowski sum is by Proposition~\ref{prop:weighted-cayley} the reflexive polytope $\sum_{i=0}^l k_iP_i$. Recall that the expanded family contains $k_0 + \cdots + k_l=r$ lattice polytopes (not necessarily distinct), where $r$ is the codegree of $P$.

\begin{example}
    Let $P_0 := \conv(0,e_1)$ and $P_1 := \conv(0,2e_2)$. Then $2P_0+P_1=[0,2]^2$ is reflexive (up to translation) and $\dim(P_1) < 2$. By Proposition~\ref{prop:weighted-cayley} $P_0 * P_1$ is a Gorenstein polytope (a three-dimensional tetrahedron) with Cayley presentation of type $(2,1)$. The expanded family is $P_0,P_0,P_1$.
    \label{reducible-example}
\end{example}

Let us say that a lattice polytope is {\em hollow} if it does not contain any lattice points in its relative interior.

\begin{lemma}\label{lem:hollow}
Let $P=P_0\cayley\cdots\cayley P_l$ be an irreducible Gorenstein polytope,
and assume that every $P_0, \ldots, P_l$ has positive dimension. Then every proper nonempty Minkowski subsum of the expanded family is hollow.
\end{lemma}

\begin{proof}
Let $r$ be the codegree of $P$. We denote by $Q_1, \ldots, Q_r$ the associated expanded family. 
Suppose that a proper subsum $Q_I := \sum_{i \in I} Q_i$ (for $\emptyset \not= I \subsetneq \{1, \ldots, r\}$) contains a relative interior lattice point $x_1$, and let $Q_J$ be the sum of the complementary subfamily. Since $Q_I+Q_J$ is reflexive (up to lattice translation), \cite[Propositions~6.11 and~6.13]{BN08} show that $Q_J$ also has a relative interior lattice point $x_2$ such that $Q_I-x_1$ and $Q_J-x_2$ are reflexive (with relative interior lattice points $0$) in complementary vector spaces:
\begin{equation}\label{eq:direct-sum}
  \lin(Q_I-x_1)\oplus_{\R}\lin(Q_J-x_2)=M_{\R}.
\end{equation}
For $k \in \{0,\ldots, l\}$ the polytope $P_k$ cannot appear in $I$ and in $J$ at the same time, i.e., $P_k=Q_i=Q_j$ for $i \in I$ and $j \in J$, as otherwise the positive-dimensional space $\lin(P_k-P_k)$ would lie in both summands of
\eqref{eq:direct-sum}.  Thus every multiplicity block is entirely contained
in $I$ or entirely contained in $J$.

Hence, all Cayley factors $P_0, \ldots, P_l$ split in two parts, say, $P_0, \ldots, P_a$ and $P_{a+1}, \ldots, P_l$, with $Q_I=\sum_{i=0}^a k_i P_i$ and $Q_J=\sum_{j=a+1}^l k_j P_j$ being reflexive. If, say, $k_0 \ge 2$, then by \cref{prop:weighted-cayley} $\dim((\sum_{i=1}^a P_i)-x_1+Q_J-x_2) < \dim(M_{\R}) = \dim(Q_I-x_1+Q_J-x_2)$. Hence \eqref{eq:direct-sum} implies $\dim(\sum_{i=1}^a P_i) < \dim(Q_I)=\dim(\sum_{i=0}^a P_i)$. So, again by \cref{prop:weighted-cayley} $P_0 * \cdots * P_a$ is Gorenstein with codegree $k_0 + \cdots + k_a$. In the same way, we get $P_{a+1} * \cdots * P_l$ is Gorenstein with codegree $k_{a+1} + \cdots + k_l$. Hence, if $F$ and $G$ are the corresponding Cayley faces of $P$, then $P$ is a degree join (as the codegrees add up) and a Cayley join (this follows from \eqref{eq:direct-sum}). Hence, $P$ is a Gorenstein join of $F$ and $G$, contradicting irreducibility.
\end{proof}

Note that Example~\ref{reducible-example} shows that irreducibility is necessary.

In order to exploit this result in what follows, the notion of a mixed degree and mixed codegree introduced in \cite{Nil20} will become useful. For a family $P_1,\ldots,P_m$ of lattice polytopes its {\em mixed codegree} is the smallest cardinality
of a nonempty subfamily whose Minkowski sum is not hollow; it is $m+1$ if all
nonempty subsums are hollow.  Its {\em mixed degree} is
\[
  \mdeg(P_1,\ldots,P_m)
  :=\dim(P_1+\cdots+P_m)+1-\mcodeg(P_1,\ldots,P_m).
\]
For us, the important fact is that the mixed degree is nonnegative \cite[Proposition~4]{Nil20}.

\subsection{Improving Cayley presentations for irreducible Gorenstein polytopes}

We will now use the previous considerations and apply them to the main objects of our interest: Gorenstein polytopes with negative Calabi-Yau dimension $q$ at least $1$. As we would like to show, they can only be irreducible if $q=1$. The following lemma is the crucial step in proving this. The essential new ingredient here is to exploit the rarely used fact from \cite{BN08} that projecting reflexive polytopes along a Minkowski summand still yields reflexive polytopes.

\begin{lemma}\label{lem:descent}
Let $P$ be an irreducible Gorenstein polytope of dimension $d$, degree $s$, codegree $r$, and with negative Calabi-Yau dimension $q\geq1$.  Suppose that $P$ has a Cayley presentation of type $(a,1^{(m)})$, where $m\geq1$. Then $a\geq q$, and one of
the following holds:
\begin{enumerate}
  \item $P$ has a Cayley presentation with all $r$ weights equal to one;
  \item $P$ has a Cayley presentation of type $(b,1^{(r-b)})$ for some integer
  \[
    q\leq b\leq a-q.
  \]
\end{enumerate}

\end{lemma}

\begin{proof}
Let \begin{equation}\label{eq:presentation}
  P=A\cayley B_1\cayley\cdots\cayley B_m \subset \MRR \oplus_\R \R^{m+1}
\end{equation} be the Cayley presentation of type $(a,1^{({m})})$. $P$ is a lattice polytope with respect to the lattice $M \oplus_\Z \Z^{m+1}$, where we may assume $n := \dim(\MRR)=\dim(A+B_1+\cdots+B_m)$. A point factor in \eqref{eq:presentation} would make $P$ a lattice pyramid.
Thus all factors have positive dimension. Since $d=n+m$, $a+m=r$, and $r=d+1-s$, we obtain
\begin{equation}\label{eq:n}
  n=s+a-1.
\end{equation}

As translating the factors by lattice points in $\MRR$ leads to isomorphic Cayley polytopes, we may assume that each $A, B_1, \ldots, B_m$ contains $0$. 
Set
\[
  S:=B_1+\cdots+B_m,
  \qquad
  U:=\lin(S),
  \qquad
  u:=\dim(U) \ge 1.
\]
Every nonempty subsum of $B_1,\ldots,B_m$, including $S$, is a proper
subsum of the expanded family.  It is hollow by \cref{lem:hollow}.  Hence
\[
  \mcodeg(B_1,\ldots,B_m)=m+1,
\]
and nonnegativity of the mixed degree gives
\begin{equation}\label{eq:mleu}
  m\leq u.
\end{equation}

Let us consider the case $a=1$. Thus, $n=s$ by \eqref{eq:n}. Hence, $r=m+1 \le u+1 \le n+1=s+1$. Hence, $q=r-s \le 1=a$ and we are in case (i).

From now on, we have $a \ge 2$. Condition (ii) in \cref{prop:weighted-cayley} gives $u<n$. 

Let
\[
  \pi:M\longrightarrow M':=M/(M\cap U)
\]
be the lattice projection along the subspace $U$. We put $A':=\pi(A)$. Since $A+S$ is $n$-dimensional, $A'$ is
full-dimensional, so $\dim(A')=n-u > 0$. The image $\pi(S)$ is a
lattice point, say $c$.  By our assumption the polytope
$aA+S$ is reflexive up to translation, and
\[
  \pi(aA+S)=aA'+c.
\]
Projection along a Minkowski summand preserves reflexivity
\cite[Corollary~6.7]{BN08}; therefore $aA'$ is reflexive up to
translation.  Thus $A'$ is Gorenstein of
codegree $a$ and with degree

\begin{equation}\label{eq:b}
  b:=\deg A'=(n-u)+1-a=s-u.
\end{equation}
As the degree is nonnegative, we get $u\leq s$.  Together with \eqref{eq:mleu}, this yields $m \le s$, so
\begin{equation}\label{eq:ageq}
  a=r-m\geq r-s=q
\end{equation}
and
\begin{equation}\label{eq:drop}
  0\leq b=s-u\leq s-m=s-(r-a)=a-q.
\end{equation}

In this situation, the idea is to apply the Cayley result to $A'$. Note that the negative CY-dimension of $A'$ equals $a-b$, where \eqref{eq:drop} implies $a-b\geq q\ge 1$. 

We distinguish two cases. 

In the case $b=0$ we have a lattice projection $M \oplus_\Z \Z^{m+1} \to M' \oplus_\Z \Z^{m+1}$ mapping $P$ to 
\[\conv(A'\times\{e_0\},\{0\}\times\{e_1\}, \ldots, \{0\} \times \{e_m\}),\]
which is an $m$-fold lattice pyramid over the unimodular $(a-1)$-simplex $A'$, thus, a unimodular simplex itself of dimension $a-1+m=r-1$. Thus, $(i)$ holds.

So, we can assume $b>0$ and Lemma~\ref{lem} yields a Cayley presentation of $A'$ of type
\begin{equation}\label{eq:inner-type}
  (b,1^{(a-b)}).
\end{equation}
This gives a lattice projection $A' \times \{e_0\} \twoheadrightarrow \conv(\tilde{e}_1,\ldots,\tilde{e}_{a-b+1})$, where $\tilde{e}_1, \ldots, \tilde{e}_{a-b+1}$ is a lattice basis of $\Z^{a-b+1}$. Here, \eqref{eq:inner-type} means that when $x'$ is the unique interior lattice point of $aA'$, we have that the canonical point $(x',a)$ of $\R_{\ge0} ( A' \times \{e_0\})$ gets mapped to $(b,1^{(a-b)})$. Here, we denote for a reflexive Gorenstein cone $\sigma$ the special point $m_{\sigma^\vee}$ its {\em canonical point}.

Summing up, we have induced lattice projections
\[M \oplus_\Z \Z^{m+1} \to M' \oplus_\Z \Z^{m+1} \to \Z^{a-b+1} \oplus_\Z \Z^m,\]
mapping 
\[A * B_1 * \cdots * B_m \twoheadrightarrow A' * \{0\} * \cdots * \{0\} 
\twoheadrightarrow \conv(\tilde{e}_1,\ldots,\tilde{e}_{a-b+1},e_1,\ldots,e_m).\]
Let us determine the type of this new Cayley presentation. Because the old type was $(a,1^{(m)})$, we know that the canonical point of the cone $\R_{\ge0} (A * B_1 * \cdots * B_m)$ can be written as $(x,a,1^{(m)}) \in M \oplus_\Z \Z^{m+1}$, where $x$ is the unique interior lattice point of $aA+S$. We know that $x$ projects to the unique interior lattice point $x'$ of $aA'$. This means that $(x,a,1^{(m)})$ gets mapped to $(x',a,1^{(m)})$, which is the canonical point of $\R_{\ge0} (A' * \{0\} * \cdots * \{0\})$. Now, this point gets mapped to $(b,1^{(a-b)},1^{(m)})=(b,1^{r-b})$, as desired. Note that $b \ge q$ holds by the arguments leading to \eqref{eq:ageq} applied to this new Cayley presentation.
\end{proof}

\begin{corollary}
Let $P$ be an irreducible Gorenstein polytope of dimension $d$ with codegree $r$ and negative CY-dimension $q \ge 1$. Then $P$ is centrally thin (i.e., $q=1$) and $P$ has a Cayley decomposition into $r$ factors.\label{irreducible}
\end{corollary}

\begin{proof}
Let $s$ be the degree of $P$. We may assume $s > 0$, in particular, $d>0$. As $q\ge1$ Lemma~\ref{lem} yields a Cayley presentation
of type $(a, 1^{(m)})$, where $m \ge 1$. Then Lemma~\ref{lem:descent} is satisfied, and thus we get a Cayley presentation of type $(1^{(r)})$ or $(b,1^{(r-b)})$ with $1 \le b < a$. Doing this successively, we must arrive at an all-unit Cayley presentation with $r$ Cayley factors, say, $P_0, \ldots P_{r-1}$, with $\dim(P_0 + \cdots + P_{r-1})=d+1-r=s$. 

If one of the Cayley factors would be a lattice point, $P$ would be a lattice pyramid, a contradiction to irreducibility. Thus, every factor has positive dimension. Now, Lemma~\ref{lem:hollow} yields that the mixed codegree of $P_0, \ldots, P_{r-1}$ is at least $r$, hence, the mixed degree is at most $s+1-r$. Nonnegativity of the mixed degree implies $0 \le s+1-r$, so, $q=r-s\le 1$, as desired. 
\end{proof}

\subsection{Properties of centrally thin Gorenstein polytopes}

\label{subsec:centrally thin}

We can now show the characterization of centrally thin Gorenstein polytopes. 

\begin{proof}[Proof of Proposition~\ref{tuples}]

Let $P_0, \ldots, P_s$ be lattice polytopes in $\R^s$ with the given properties. In the notation of \cite{batyrevborisov,BN08} this means that $P_0, \ldots, P_s$ is a proper centered irreducible nef-partition of length $s+1$ in $\R^s$. By \cite[Theorem~5.8]{NS13}, $P := P_0 * 
\cdots * P_s$ is an irreducible Gorenstein polytope of dimension $2s$, codegree $s+1$ and degree $s$, so $P$ is centrally thin. On the other hand, let $P$ be a centrally thin Gorenstein polytope of dimension $2s$, codegree $s+1$ and degree $s > 0$. By Corollary~\ref{irreducible}, $P$ has a Cayley decomposition into $s+1$ factors. So, $P \cong P_0 * \cdots * P_s$ for $P_0, \ldots, P_s$ lattice polytopes in $\R^s$ and $Q := P_0 + \cdots + P_s$ is reflexive (with a unique interior lattice point $x \in \R^s$). As $P$ would otherwise be a lattice pyramid, and thus reducible, each factor $P_0, \ldots, P_s$ has positive dimension. By duality, $P^\times$ is also a centrally thin Gorenstein polytope of same dimension, codegree, degree and negative CY-dimension. Hence, again Corollary~\ref{irreducible} implies that $P^\times$ also has a Cayley decomposition into $s+1$ factors. By \cite[Definition~2.4 and Proposition~3.6]{BN08} we have lattice points $p_0 \in P_0, \ldots, p_s \in P_s$ such that $p_0 + \cdots + p_s = x$, thus, $P_0, \ldots, P_s$ form a proper nef-partition in the notation of \cite{BN08}. Let us define lattice polytopes $P'_0 := P_0 - p_0, \ldots, P'_s := P_s - p_s$. Then $P'_0, \ldots, P'_s$ are each positive-dimensional, contain $0$ and satisfy $P'_0 + \cdots + P'_s = Q-x$, which is a reflexive polytope with $0$ in its interior. Note that $P'_0 * \cdots * P'_s \cong P_0 * \cdots * P_s$. Again, \cite[Theorem~5.8]{NS13} shows that $P'_0, \ldots, P'_s$ is an irreducible nef-partition, so all conditions $(i)$ to $(iv)$ are satisfied.
\end{proof}

In the notation of \cite{BN08}, the crucial property is here that the associated reflexive Gorenstein cones of a centrally thin Gorenstein polytope are {\em completely split}.

Let us next show the vertex-spanning property. 

\begin{proof}[Proof of Proposition~\ref{duality}]
Write $s=\deg P$. We may assume $s > 0$, as otherwise $P$ would just be a unimodular simplex, which is clearly vertex-spanning.
We may assume that $P = P_0* \cdots *P_s \subset \R^s \oplus \R^{s+1}$ with $P_0, \ldots, P_s \subset \R^s$ as in Proposition~\ref{tuples}. It follows by Remark~\ref{strengthen} from property (iii)' that the mixed codegree of $P_1, \ldots, P_s$ is $s+1$. Hence, $P_1, \ldots, P_s$ forms a family of lattice polytopes of mixed degree $\dim(P_1 + \cdots + P_s) + 1 - (s+1)$. As $\dim(P_1 + \cdots + P_s) \le s$ and the mixed degree is nonnegative, we get that $P_1, \ldots, P_s$ has mixed degree $0$. By \cite[Theorem~5]{Nil20} and property (i), this implies that 
\begin{equation}\label{eq:mv-one}
  \MV(P_1,\ldots,P_s)=1.
\end{equation}

For $1\leq i\leq s$, set
\[
  D_i:=\{y-x:x,y\in\vertices(P_i)\}\subset \Z^s.
\]
Since the mixed volume is positive, it is well-known (see \cite[Theorem~5.1.8]{Sch93}) that we may therefore choose
linearly independent vectors $w_1, \ldots, w_s \in \Z^s$ with
\[
  w_i=y_i-x_i\in D_i,
  \]
where $x_i,y_i \in \vertices(P_i)$ for $i=1, \ldots, s$, so the segments $[x_i,y_i]$ are contained in $P_i$.  Monotonicity of mixed
volume and \eqref{eq:mv-one} yield
\[
  1
  \leq |\det(w_1,\ldots,w_s)|
  =\MV([x_1,y_1],\ldots,[x_s,y_s])
  \leq1.
\]
Thus $w_1,\ldots,w_s$ form a lattice basis of $\Z^s$.

We have to show that every lattice point in $\Z^s \oplus \Z^{s+1}$ is in the subgroup $\Gamma$ generated by $\{(z_i,e_i) \,:\, z_i \in \vertices(P_i),\; i \in \{0, \ldots, s\}\}$. 

We observe that
\[
  (w_i,0)=(y_i,e_i)-(x_i,e_i) \in \Gamma,
  \qquad 1\leq i\leq s,
\]
so by the lattice basis property we get $\Z^s \oplus \{0\}\subset \Gamma$. 

Choose a vertex $z_i$ of $P_i$ for $i=0, \ldots, s$. Then
\[
  (0,e_i)=(z_i,e_i)-(z_i,0) \in \Gamma
  \qquad 0\leq i\leq s.
\]
Thus, $\Gamma = \Z^s \oplus \Z^{s+1}$, so $P$ is vertex-spanning. As $P^\times$ is also centrally thin, $P^\times$ is also vertex-spanning.

\end{proof}
\section{From Gorenstein joins to free joins}\label{sec:free-joins}

\subsection{Associativity and splitting of degree joins}
\label{subsec:splitting}

Let us prove the statements on degree joins mentioned above. Let us first make a general observation. Let $P$ be a join of faces $F_0, \ldots, F_l$, equivalently, 
\begin{equation}C_P=C_{F_0}\oplus_{\mathbb R}\cdots\oplus_{\mathbb R}C_{F_l}.
\label{1}
\end{equation}In particular, this implies
\begin{equation}
\codeg P\leq\sum_{i=0}^l\codeg F_i,
\end{equation} and thus,
\begin{equation}
\deg P\geq\sum_{i=0}^l\deg F_i.
\label{2}
\end{equation}

\begin{proof}[Proof of Proposition~\ref{associative}]

Suppose (i) holds. Ordinary joins may be regrouped, so $G=\conv(F_1,\ldots,F_l)$  is a join of $F_1,\ldots,F_l$, while \(P\) is a join of \(F_0\) and \(G\). Applying \eqref{2} twice yields
\[
\deg P
\geq \deg F_0+\deg G
\geq \deg F_0+\sum_{i=1}^l\deg F_i.
\] Hence, overall equality implies (ii). Conversely, associativity of joins implies that $P$ is a join of  $F_0,\ldots,F_l$ and our assumptions imply that both inequalities in \eqref{2} are equalities, so we get (i).

Finally, suppose that \(P\) is Gorenstein. For $i=0, \ldots,l$ choose any lattice point $m_i\in\relint(C_{F_i})$ on height $\codeg F_i$. Then by our assumption
\[
m:=m_0+\cdots+m_l\in\relint(C_P)
\]has height \(\codeg P\). Since \(P\) is Gorenstein, \(m=m_{{C_P}^\vee}\).

Fix \(i \in \{0, \ldots, l\}\), and let \(D\) be a facet of \(C_{F_i}\). Then
\[
D\oplus_{\mathbb R}\bigoplus_{j\neq i}C_{F_j}
\]is a facet of \(C_P\). Let \(u\) be the corresponding vertex of $P^\times$. Hence, $u':=u|_{C_{F_i}}$ defines an inner normal of $D$. By definition we have
\begin{equation}
\langle m,u\rangle=1.
\label{3}
\end{equation}The functional \(u\) vanishes on every \(C_{F_j}\) with \(j\neq i\). So, we get $1=\langle m,u\rangle=\langle m_i,u' \rangle$. Therefore, $u'$ is even a primitive integral functional. Therefore, $m_i$ has lattice distance one from the facet $D$. As this is true for every facet of $C_{F_i}$, this shows that $(\codeg F_i) F_i$ is reflexive, respectively, $F_i$ is Gorenstein.
\end{proof}

\begin{proof}[Proof of Lemma~\ref{splitting}]
By the previous result, $P$ is a degree join of $F_0$ and $G$, where $F_0$ and $G$ are Gorenstein. 
Let us abbreviate $\sigma := C_P$, $\sigma_0 := C_{F_0}$, and $\sigma_G := C_G$ . We also define $\overline M_0$ as the restriction of $\overline M$ to $\lin \sigma_0$, and $\overline M_G$ by restricting to $\lin \sigma_G$.

Let $u\in\overline N$ be a primitive ray generator of $\sigma^{\vee}$ that evaluates $0$ on $\sigma_G$. As in the previous proof restricting $u$ as an integral functional $u'$ on $\overline M_0$ gives an inner facet normal of $\sigma_0$ that is primitive in the lattice $\overline N_0$ dual to $\overline M_0$. Hence, $u'$ is a vertex of $F_0^\times$. As any vertex of $F_0^\times$ occurs in this way and as $F_0^\times$ is vertex-spanning (with respect to $\overline N_0$), this implies that restricting to $\lin \sigma_0$ yields an isomorphism
\begin{equation}
    \overline N \cap \lin(\sigma_G)^\perp \cong \overline N_0,
    \label{perpy}
\end{equation}
which we write as $u \mapsto u'$. From this the desired splitting $\overline M_0 \oplus \overline M_G = \overline M$ follows directly. Let us write this up explicitly. Let $x\in\overline M$. There is a unique decomposition
\[
  x=x_0+x_G,
  \qquad
  x_0\in \lin \sigma_0,
  \quad
  x_G \in \lin \sigma_G
\]
Let $u'\in\overline N_0$. From \eqref{perpy} we get
\[
  \langle u',x_0\rangle=\langle u,x\rangle\in\Z.
\]
Hence, $x_0\in \overline M_0$, and thus also $x_G \in \overline M$, so in $\overline M_G$.
\end{proof}

\subsection{Proof of the main theorem}

We can now put everything together.

\begin{proof}[Proof of Theorem~\ref{main}]

Let $P$ be a Gorenstein polytope of dimension $d$ and with negative CY-dimension $q \ge 1$. We recursively decompose $P$ using Gorenstein joins. If $P$ is irreducible or a lattice point, then Corollary~\ref{irreducible} implies that $P$ is also centrally thin, so we stop. Note that this case can only happen for $q=1$. If $P$ is reducible, then write $P$ as a Gorenstein join of two faces (each necessarily also Gorenstein). Do the same repeatedly for each factor. As the dimensions drop, this process will terminate. Note as Gorenstein joins are degree joins and the degree join property is associative by Proposition~\ref{associative}, this leads to $P$ being a degree join of Gorenstein faces $P_1, \ldots, P_k$ with $k \ge 2$ where each $P_i$ is irreducible or a lattice point. Write $q_i$ for the negative CY-dimension of $P_i$ for $i=1, \ldots, k$. By Corollary~\ref{irreducible} (or directly for lattice points) we have $q_i \le 1$ for each $i=1, \ldots, k$. Therefore, as $q=q_1+ \cdots+q_k \ge 1$, there must exist at least one $q_i =1$. W.l.o.g. let $q_1=1$, so $P_1$ is centrally thin. By Proposition~\ref{duality} and Lemma~\ref{splitting}, $P$ is a free join with faces $P_1$ and $P' := \conv(P_2, \ldots, P_k)$, where $P'$ is Gorenstein with negative CY-dimension $q' = q-1$. If $q' \ge 1$, repeat the argument for $P'$, otherwise, we have $q'=0$ and stop. 

For the last statement in the theorem note that free joins of Gorenstein polytopes are Gorenstein (so in case (ii) one can combine the last two factors).
\end{proof}

\bibliographystyle{amsplain}
\bibliography{references}

@incollection{BN08,
  author    = {Batyrev, Victor V. and Nill, Benjamin},
  title     = {Combinatorial Aspects of Mirror Symmetry},
  booktitle = {Integer Points in Polyhedra---Geometry, Number Theory, Representation Theory, Algebra, Optimization, Statistics},
  series    = {Contemp. Math.},
  volume    = {452},
  pages     = {35--66},
  publisher = {Amer. Math. Soc.},
  address   = {Providence, RI},
  year      = {2008}
  }

@article{HNP09,
  author  = {Haase, Christian and Nill, Benjamin and Payne, Sam},
  title   = {Cayley Decompositions of Lattice Polytopes and Upper Bounds for {$h^*$}-Polynomials},
  journal = {J. Reine Angew. Math.},
  volume  = {637},
  pages   = {207--216},
  year    = {2009},
  doi     = {10.1515/CRELLE.2009.096}
}

@mastersthesis{Knu25,
  author = {Knupfer, Johannes},
  title  = {Der {G}leichheitsfall in der {Batyrev--Juny}-{V}ermutung \"uber {Gorenstein}-{P}olytope},
  school = {Otto-von-Guericke-Universit\"at Magdeburg},
  year   = {2025},
  type   = {Master's thesis}
}

@article{Nil24,
  author  = {Nill, Benjamin},
  title   = {Proof of a Conjecture of {Batyrev} and {Juny} on {Gorenstein} Polytopes},
  journal = {Discrete Comput. Geom.},
  volume  = {72},
  number  = {4},
  pages   = {1519--1529},
  year    = {2024},
  doi     = {10.1007/s00454-023-00575-0}
}

@article{NS13,
  author  = {Nill, Benjamin and Schepers, Jan},
  title   = {{Gorenstein} Polytopes and Their Stringy {$E$}-Functions},
  journal = {Math. Ann.},
  volume  = {355},
  number  = {2},
  pages   = {457--480},
  year    = {2013},
  doi     = {10.1007/s00208-012-0792-2}
}

@misc{preserving,
 author = {Vadym Kurylenko and Benjamin Nill},
 title = {Preserving {Hodge} {Vectors} of {Lattice} {Polytopes}},
 year = {2026},
 howpublished = {Preprint, {arXiv}:2602.20765 [math.{CO}] (2026)},
 url = {https://arxiv.org/abs/2602.20765},
 arXiv = {arXiv:2602.20765}
}

@article{borger_thin_2023,
	title = {Thin {Polytopes}: {Lattice} {Polytopes} {With} {Vanishing} {Local} \textit{h} *-{Polynomial}},
	issn = {1073-7928, 1687-0247},
	shorttitle = {Thin {Polytopes}},
	url = {https://academic.oup.com/imrn/advance-article/doi/10.1093/imrn/rnad231/7286974},
	doi = {10.1093/imrn/rnad231},
	language = {en},
	urldate = {2024-02-25},
	journal = {International Mathematics Research Notices},
	author = {Borger, Christopher and Kretschmer, Andreas and Nill, Benjamin},
	month = oct,
	year = {2023},
	pages = {rnad231},
}

@article{batyrev_multiples_2007,
	title = {Multiples of lattice polytopes without interior lattice points},
	volume = {7},
	issn = {1609-3321,1609-4514},
	url = {https://doi.org/10.17323/1609-4514-2007-7-2-195-207},
	doi = {10.17323/1609-4514-2007-7-2-195-207},
	number = {2},
	journal = {Moscow Mathematical Journal},
	author = {Batyrev, Victor and Nill, Benjamin},
	year = {2007},
	mrnumber = {2337878},
	pages = {195--207, 349},
}

@article{batyrev_mirror_1996,
	title = {Mirror duality and string-theoretic {Hodge} numbers},
	volume = {126},
	issn = {0020-9910,1432-1297},
	url = {https://doi.org/10.1007/s002220050093},
	doi = {10.1007/s002220050093},
	number = {1},
	journal = {Inventiones Mathematicae},
	author = {Batyrev, Victor V. and Borisov, Lev A.},
	year = {1996},
	pages = {183--203}
}

@article{borisov_string_2003,
	title = {String cohomology of {Calabi}–{Yau} hypersurfaces via mirror symmetry},
	volume = {180},
	issn = {0001-8708},
	url = {https://www.sciencedirect.com/science/article/pii/S0001870803000070},
	doi = {10.1016/S0001-8708(03)00007-0},
	number = {1},
	urldate = {2023-09-25},
	journal = {Advances in Mathematics},
	author = {Borisov, Lev A. and Mavlyutov, Anvar R.},
	month = dec,
	year = {2003},
	pages = {355--390},
}

@book{GKZ,
	series = {Mathematics: {Theory} \& {Applications}},
	title = {Discriminants, resultants, and multidimensional determinants},
	isbn = {0-8176-3660-9},
	url = {https://doi.org/10.1007/978-0-8176-4771-1},
	publisher = {Birkhäuser Boston, Inc., Boston, MA},
	author = {Gelfand, I. M. and Kapranov, M. M. and Zelevinsky, A. V.},
	year = {1994},
	doi = {10.1007/978-0-8176-4771-1},
}

@article {BJ10,
    AUTHOR = {Batyrev, Victor and Juny, Dorothee},
     TITLE = {Classification of {G}orenstein toric del {P}ezzo varieties in
              arbitrary dimension},
   JOURNAL = {Mosc. Math. J.},
  FJOURNAL = {Moscow Mathematical Journal},
    VOLUME = {10},
      YEAR = {2010},
    NUMBER = {2},
     PAGES = {285--316, 478},
      ISSN = {1609-3321},
   MRCLASS = {14M25 (14J45 52B20)},
  MRNUMBER = {2722799},
MRREVIEWER = {Jaros\l aw A. Wi\'{s}niewski},
       DOI = {10.17323/1609-4514-2010-10-2-285-316},
       URL = {https://doi.org/10.17323/1609-4514-2010-10-2-285-316},
}

@book {Sch93,
    AUTHOR = {Schneider, Rolf},
     TITLE = {Convex bodies: the {B}runn-{M}inkowski theory},
    SERIES = {Encyclopedia of Mathematics and its Applications},
    VOLUME = {44},
 PUBLISHER = {Cambridge University Press},
   ADDRESS = {Cambridge},
      YEAR = {1993},
     PAGES = {xiv+490}
}

@article {Nil20,
    AUTHOR = {Nill, Benjamin},
     TITLE = {The mixed degree of families of lattice polytopes},
   JOURNAL = {Ann. Comb.},
  FJOURNAL = {Annals of Combinatorics},
    VOLUME = {24},
      YEAR = {2020},
    NUMBER = {1},
     PAGES = {203--216},
      ISSN = {0218-0006},
   MRCLASS = {52B20 (11H06)},
  MRNUMBER = {4078146},
MRREVIEWER = {G\'{a}bor Hetyei},
       DOI = {10.1007/s00026-019-00490-3},
       URL = {https://doi.org/10.1007/s00026-019-00490-3},
}

@preamble{
   "\def\cprime{$'$} "
}

@Book{F,
  author = 	 {Fulton, W.},
  title = 	 {Introduction to Intersection Theory in Algebraic Geometry},
  publisher = 	 {Amer. Math. Soc.},
  year = 	 {1984},
  volume = 	 {54},
  OPTnumber = 	 {},
  series = 	 {BMS Regional Conf. Ser. in Math.},
  address = 	 {Providence}
}

@incollection{batyrevborisov,
 author = {Batyrev, Victor V. and Borisov, Lev A.},
 title = {On {Calabi}-{Yau} complete intersections in toric varieties},
 booktitle = {Higher dimensional complex varieties. Proceedings of the international conference, Trento, Italy, June 15--24, 1994},
 isbn = {3-11-014503-0},
 pages = {39--65},
 year = {1996},
 publisher = {Berlin: Walter de Gruyter},
 language = {English},
 zbMATH = {1054910},
 Zbl = {0908.14015}
}

@article{batyrev94,
 author = {Batyrev, Victor V.},
 title = {Dual polyhedra and mirror symmetry for {Calabi}-{Yau} hypersurfaces in toric varieties},
 fjournal = {Journal of Algebraic Geometry},
 journal = {J. Algebr. Geom.},
 issn = {1056-3911},
 volume = {3},
 number = {3},
 pages = {493--535},
 year = {1994},
 language = {English},
 zbMATH = {653320},
 Zbl = {0829.14023}
}

\end{document}